\documentclass[12pt,draft]{amsart}
\usepackage{amssymb,latexsym, amsmath, amscd, array, graphicx}

\newtheorem{thm}{Theorem}[section]
\newtheorem{theorem}[thm]{Theorem}

\newtheorem{lemma}[thm]{Lemma}

\newtheorem{prop}[thm]{Proposition}

\theoremstyle{definition}

\theoremstyle{definition}

\newtheorem{definition}[thm]{Definition}

\newtheorem{remark[thm]}{Remark}

\newtheorem{cor}[thm]{Corollary}

\theoremstyle{remark}

\DeclareMathOperator{\cd}{{\rm cd}}

\DeclareMathOperator{\cat}{{\mbox{\rm cat$_{\rm LS}$}}}

\def\Ord{\protect\operatorname{Ord}}

\def\Wi{\widetilde}

\def\cd{\protect\operatorname{cd}}

\def\scr{\mathcal}

\def\C{{\mathbb C}}

\long\def\forget#1\forgotten{} %

\numberwithin{equation}{section}

\begin{document}

\title[Lusternik-Schnirelmann category]{
An upper bound on the LS-category in presence of the fundamental group}
\author[A.~Dranishnikov]{Alexander N. Dranishnikov}


\address{Department of Mathematics, University
of Florida, 358 Little Hall, Gainesville, FL 32611-8105, USA}
\curraddr{}
\email{dranish@math.ufl.edu}


\date{}

\dedicatory{}

\commby{Daniel Ruberman}

\begin{abstract}
We prove that
$$
\cat X\le \frac{\cd(\pi_1(X))+\dim X}{2}$$
for every CW complex $X$ where $\cd(\pi_1(X))$ denotes the cohomological
dimension of the fundamental group of $X$. We obtain this as a corollary of the  inequality
$$
\cat X\le\frac{\cat (u_X)+\dim X}{2}
$$
where $u_X:X\to B\pi_1(X)$ is a classifying map for the universal covering of $X$.
\end{abstract}

\maketitle

\section{Introduction} The reduced {\em Lusternik-Schnirelmann category}
(briefly LS-category) $\cat X$ of a topological space $X$ is the
minimal number $n$ such that there is an open cover $\{U_0,\dots,
U_n\}$ of $X$ by $n+1$ contractible in $X$ sets. We note that the LS-category 
is a homotopy invariant. The Lusternik-Schnirelmann category has many applications.
Perhaps the most famous is the classical Lusternik and Schnirelmann theorem
\cite{CLOT} which states that $\cat M$ gives a low bound
for the number of critical points on a  manifold $M$ of any smooth not
necessarily Morse function. This theorem was used by Lusternik and Schnirelmann
in their solution of Poincare's problem on the existence of three closed geodesics on a 2-sphere~\cite{LS}. In modern time the LS-category was used in the proof of the Arnold conjecture on symplectomorphisms~\cite{Ru}.

The LS-category is a numerical homotopy invariant which is difficult to compute.
Even to get a reasonable bound for $\cat $ very often is a serious problem.
In this paper we discuss only upper bounds.
For nice spaces, such as CW complexes,
it is an easy observation that $\cat X\le\dim X$. In the 40s
Grossmann~\cite{Gro} (and independently in the 50s G.W.
Whitehead~\cite{Wh}~\cite{CLOT}) proved that for simply connected CW
complexes $\cat X\le\dim X/2$. 

In the presence of the fundamental group the LS-category can be
equal to the dimension. In fact, $\cat X=\dim X$ if and only if $X$ is essential in 
the sense of Gromov. This was proven for manifolds in~\cite{KR}. For general CW complexes we refer to Proposition~\ref{essential} of this paper. We recall that an $n$-dimensional complex $X$ is called {\em inessential} if a map $u_X:X\to B\pi_1(X)$ that classifies its universal cover can be deformed to the $(n-1)$-skeleton $(B\pi_1(X))^{(n-1)}$. Otherwise, it is called {\em essential}. Typical examples of essential CW complexes are aspherical manifolds.

Yu.~Rudyak conjectured that in the case of free
fundamental group there should be the Grossmann-Whitehead type
inequality at least for closed manifolds. There were partial results
towards Rudyak's conjecture \cite{DKR},\cite{St} until it was
settled in \cite{Dr1}. Later it was shown in \cite{Dr2} (also see the followup~\cite{OS})
that the
Grossmann-Whitehead type estimate holds for complexes with the
fundamental group having small cohomological dimension.  Namely, it was shown
that $\cat X\le \cd(\pi_1(X))+\dim X/2$.

Clearly, this upper bound is far from being optimal for fundamental groups with sufficiently large cohomological dimension. Indeed, for the product of an aspherical $m$-manifold $M$ 
with the complex projective space we have
$\cat(M\times\C P^n)=m+n$ but our upper bound is $m+(m+2n)/2=\frac{3}{2}m+n$.
Moreover, our bound quits to be  useful for complexes with $\cd(\pi_1(X))\ge\dim X/2$.
The desirable bound here is $$\cat X\le \frac{\cd(\pi_1(X))+\dim X}{2}.$$
Such an upper bound was proven in~\cite{DKR} for the systolic category, a differential geometry relative of the LS-category.
Nevertheless, for the classical LS-category a similar estimate was missing until now.

In this paper we prove the desirable upper bound. We obtain such a bound as a corollary of the following inequality
$$
\cat X\le\frac{\cat (u_X)+\dim X}{2}
$$
where $u_X:X\to B\pi_1(X)$ is a classifying map for the universal covering of $X$.
We note that this inequality gives a meaningful upper bound on the LS-category for complexes
with any fundamental group. Also we note that the new upper bound gives the optimal estimate
for the above example  $M\times\C P^n$, the product of an aspherical manifold and the complex projective space. Namely, $$\cat(M\times\C P^n)\le(m+(m+2n))/2=m+n.$$

\section{Preliminaries}
The proof of the new upper bound for $\cat X$ is based on a further modification of the Kolmogorov-Ostrand multiple cover technique~\cite{Dr1}. That technique was extracted by Ostrand from the work of Kolmogorov on the 13th Hilbert problem~\cite{Os}.
Also in this paper we make use of the following well-known fact.
\begin{prop}\label{domination}
	Let $f:X\to Y$ be a homotopy domination. Then $\cat Y\le \cat X$.
\end{prop}
\begin{proof}
Let $s:Y\to X$ be a left homotopy inverse to $f$, i.e. $f\circ s\sim 1_Y$. Let $U_0,\dots, U_k$ be
an open cover of $X$ by sets contractible in $X$. One can easily check that
$s^{-1}(U_0),\dots,s^{-1}(U_k)$ is an open cover by sets contractible in $Y$.	
	\end{proof}

Let $\scr
U=\{U_{\alpha}\}_{\alpha\in A}$ be a family of sets in a topological
space $X$. 
The {\em multiplicity}
of $\scr U$ (or the {\em order}) at a point $x\in X$, denoted
$\Ord_x\scr U$, is the number of elements of $\scr U$ that contain
$x$.  A family $\scr U$ is a cover of $X$ if
$\Ord_x\scr U\ne 0$ for all $x$.

\begin{definition} A family $\scr U$ of subsets of $X$ is called a {\em
$k$-cover}, $k\in N$ if every subfamily of $\scr U$ that consists of $k$ sets forms a
cover of $X$.
\end{definition}
The following is obvious (see \cite{Dr1}).
\begin{prop}\label{n-cover}
A family $\scr U$ that consists of $m$ subsets of $X$ is an
$(n+1)$-cover of $X$ if and only if $\Ord_x\scr U\ge m-n$ for all
$x\in X$.
\end{prop}

\

Let $K$ be a simplicial complex. By the definition the dual to the $m$-skeleton $K^{(m)}$ is a subcomplex 
$L=L(K,m)$ of the barycentric subdivision $\beta K$ that consists of simplices of $\beta K$ which do not intersect $K^{(m)}$. Note that $\beta K$ is naturally embedded in the join product
$K^{(n)}\ast L$. Then the following is obvious:

\begin{prop}\label{complement}
For any $n$-dimensional complex $K$ the complement $K\setminus K^{(m)}$ 
to the $m$-skeleton is homotopy equivalent to an $(n-m-1)$-dimensional complex $L$.
\end{prop}
\begin{proof}
The complex $L$ is the dual  to $K^{(m)}$. Clearly, $\dim L=n-m-1$. The	complement $K\setminus K^{(m)}$
can be deformed to $L$ along the  field of intervals defined by the embedding $\beta K\subset K^{(n)}\ast L$.
	\end{proof}

Let $f:X\to Y$ be a continuous map. We recall that the LS-category of $f$, $\cat f$ is the smallest number $k$ such that $X$ can be covered by $k+1$ open sets $U_0,\dots,U_k$ such that the restriction $f|_{U_i}:U_i\to Y$ of $f$ to each of them is null-homotopic. Clearly, $$\cat f\le\cat X,\cat Y.$$

We denote by $u_X:X\to B\pi$, $\pi=\pi_1(X)$, a map that classifies the universal covering
$p:\Wi X\to X$ of $X$. Thus, $p$ is the pull-back of the universal covering $q:E\pi\to B\pi$.
Here $B\pi$ is any aspherical CW complex with the fundamental group $\pi$. Thus, any map
$u:X\to B\pi$ that induces an isomorphism of the fundamental groups is a classifying map.

The following proposition is proven in~\cite{Dr3}, Proposition 4.3.
\begin{prop}\label{classifying}
	A classifying map $u_X:X\to B\pi$ of the universal covering of a CW complex $X$
	can be deformed into the $d$-skeleton $B\pi^{(d)}$ if and only if $\cat(u_X)\le d$.
\end{prop}
The following proposition for closed manifolds was proven by Katz and Rudyak~\cite{KR},
although it was already known to Berstein in a different  equivalent formulation~\cite{Be}.
\begin{prop}\label{essential}
	For an $n$-dimensional CW complex $X$, $\cat X=n$ if and only if $X$ is essential.
\end{prop}
\begin{proof}
	Suppose that $X$ is essential. By Proposition~\ref{classifying} we obtain that $\cat(u_X)>n-1$.
	Thus, $\cat X\ge\cat(u_X)\ge n$ and, since $\dim X=n$, $\cat X=n$.

	The implication in the other direction can be derived from the proof of Theorem 4.4 in~\cite{Dr3}. 
	Here we give the sketch of the proof. Let $u_X:X\to B\pi^{(n-1)}$ be  a classifying map.
	To prove the inequality $\cat X\le n-1$ it suffices to show that the Ganea-Schwarz fibration $p_n^X:G_{n-1}(X)\to X$ admits a section. Since the fiber of the Ganea-Schwarz fibration 
	$p_n^{B\pi}$ is $(n-1)$-connected, the map $u_X$ admits a lift $f:X\to G_{n-1}(B\pi)$.
	Then the map $p'$ in the pull-back diagram 
	$$
	\begin{CD}
	G_{n-1}(X) @>q>> Z @>u_X'>> G_{n-1}(B\pi)\\
	@. @Vp'VV @Vp^{B\pi}_{n-1}VV\\
	@. X @>u_X>> B\pi^{(n-1)}\\
	\end{CD}
	$$
	admits a section  $s:X\to Z$.
	 Here $p^X_{n-1}=p'\circ q$.
	Since $X$ is $n$-dimensional, to show that $s$ has a
	lift with respect to $q$ it suffices to prove that the homotopy fiber $F$ of the map $q$ is
	$(n-1)$-connected. Note that the homotopy exact sequence of the fibration
	$$F \to(p^X_{n-1})^{−1}(x_0)\stackrel{u'}\rightarrow (p^{B\pi}_{n-1})^{-1}(y_0)$$
	where $u'$ is the restriction of $u_X'\circ q$
	to the fiber $(p^X_{n-1})^{-1}(x_0)$
	coincides with the homotopy exact sequence of the fibration
	$$F′ \to \ast_{n}\Omega(M)\stackrel{\ast\Omega(u_X)}{\longrightarrow}\ast_n\Omega(B\pi)$$
	obtained from the loop map $\Omega(u_X)$ turned into a fibration by taking the iterated join product.
	Since  $\pi_0(\Omega u_X) = 0$,
	we obtain $\pi_i(\ast_n\Omega u_X) = 0$ for $i \le n-1$ (see Proposition 2.4~\cite{Dr3}) 
	and hence $\pi_i(F) = 0$ for $i \le n- 1$. 
	\end{proof}

\section{Multiple covers of polyhedra}

For a point $x\in X$ in a CW complex $X$ by $d(x)$ we denote the dimension of the open cell $e$
containing $x$.
We call a subset $A\subset X$ in a CW complex $X$ {\em $r$-deformable}
if $A$ can be deformed in $X$ to the $r$-skeleton $X^{(r)}$. A deformation $H:A\times I\to X$
to the 0-skeleton $X^{(0)}$ is called {\em monotone} if $d(H(x,t))$ is monotonically decreasing function of $t$
for all $x\in A$.
\begin{prop}\label{cover}
	Let $X$ be a connected simplicial complex of dimension $\le N(r+1)-1$. Then for  any $m\ge N$ 
	there exists an open cover  $\scr U=\{U_1,\dots,U_m\}$ of $X$ by $r$-deformable sets
	 such that $\Ord_x\scr U\ge m-k+1$ for every $k\le N$ and all $x\in X^{(k(r+1)-1)}$. Equivalently, the restriction of $\scr U$ to the $(k(r+1)-1)$-skeleton is a $k$-cover.
	 
	 Moreover, for $r=0$ we may assume that
	 each set $U_i$ is monotone $r$-deformable. 
\end{prop}
\begin{proof}
	It suffices to prove the Proposition for complexes with $\dim X=N(r+1)-1$.
	We do it by induction on $n$. For $N=1$ the statement is obvious.
	Suppose that it holds true for $N-1\ge 1$. We prove it for $N$ by induction on $m$.
	First we establish the base of induction by proving the statement for $m=N$. 
	By the external induction applied to $X^{((N-1)(r+1)-1)}$ with $m=N-1$
	there is an open cover  $\scr U=\{U_1,\dots, U_{N-1}\}$ of $X^{((N-1)(r+1)-1)}$ 
	such that each $U_i$ is $r$-deformable and $\Ord_x\scr U\ge (N-1)-k+1=N-k$ for all $x\in X^{(k(r+1)-1)}$. We can enlarge each  $U_i$ to  a $r$-deformable open in $X$ set $U_i'\subset X$.
	
	Let $G=\bigcup_{i=1}^{N-1}U'_i$.
	Since the complement $X\setminus X^{((N-1)(r+1)-1)}$ is homotopy equivalent  to a $r$-dimensional complex (see Proposition~\ref{complement}),
	$Z_0=X\setminus G$ is $r$-deformable. Since $Z_0$ is closed, we can find an open enlargement $W_0$ 
	to an $r$-deformable set whose closure does not intersect $X^{((N-1)(r+1)-1)}$ .
	Thus, the cover $\{U_1',\dots, U_{N-1}',W_0\}$ satisfies the condition of Proposition for $k=N$.
	
	Consider the set $$Z_1=\{x\in X^{(N-1)(r+1)-1)}\mid \Ord_x\scr U= 1\}.$$ Clearly, $Z_1$ is closed. 
	By the induction assumption $Z_1$ does not intersect the
	skeleton $X^{((N-2)(r+1)-1)}$. Since the complement, 
	$$
	X^{((N-1)(r+1)-1)}\setminus X^{((N-2)(r+1)-1)}
	$$
	is homotopy equivalent to an $r$-dimensional complex,
	$Z_1$ is $r$-deformable in $X^{((N-1)(r+1)-1)}$. 
	Let $W_1$ be an enlargement of $Z_1$ to an open $r$-deformable  in $X$ sets such that the closure $\bar W_1$ does not intersect
	$\bar W_0\cup X^{(N-2)(r+1)-1)}$. Note that the cover $\{U_1',\dots, U_{N-1}',W_0\cup W_1\}$ satisfies the
	condition of Proposition with $k=N$ and $k=N-1$.
	
	 Next we consider $$Z_2=\{x\in X^{((N-2)(r+1)-1)}\mid \Ord_x\scr U= 2\}$$ and similarly 
	 define an open set $W_2$ and so on up to $W_{N-1}$. By the construction
	 each set $W_i$ is $r$-deformable and the closures $\bar W_i$ are disjoint. Therefore,
	 the union $U_N'=W_0\cup\dots\cup W_{N-1}$ is $r$-contractible.
	 Then the cover $U_0',\dots, U_N'$ satisfies all the conditions of Proposition for all $k\le N$.
	
	The proof of the inductive step is very similar to the above.
	Assume that the statement of Proposition holds for $N$ and $m-1\ge N$. We prove it for $N$ and $m$. Let $\scr U=\{U_1,\dots, U_{m-1}\}$ be an open cover of $X$ by $r$-deformable sets such that for any $k\le N$ the restriction of $\scr U$ to $X^{(k(r+1)-1)}$ is a $k$-cover. Thus, $\Ord_x\scr U\ge (m-1)-N+1=m-N$ for all $x$. Let
	$$Z_0=\{x\in X\mid\Ord_x\scr U=m-N\}.$$ By the induction assumption $Z_0\cap X^{((N-1)(r+1)-1)}=\emptyset$. Thus,
	$Z_0$ is $r$-deformable in $X$. 
	We consider an open $r$-deformable neighborhood $W_0$ of $Z_0$ with $\bar W_0\cap X^{(N-1)(r+1)-1}=\emptyset$.
	
	Next we consider the closed set $$Z_1=\{x\in X^{((N-1)(r+1)-1)}\mid \Ord_x\scr U=m-N+1\}.$$ By the induction assumption $Z_1$ does not intersect $X^{((N-2)(r+1)-1)}$. As above, we define
	a $r$-deformable set $W_1$  with $$\bar W_1\cap (\bar W_0\cup X^{((N-2)(r+1)-1)})=\emptyset$$ and so on. We define $U_m=W_0\cup\dots\cup W_{N-1}$.
	Then the condition of Proposition is satisfied for all $k$ with $\scr U'=\{ U_1,\dots, U_{m-1},U_m\}$.
	
	Now we revise our proof for $r=0$ in order to verify the extra condition of Proposition. Note that $\dim X\le N-1$ in this case.
In the proof of the base of induction on $m$ the enlargements $U_i'$ can be chosen monotone deformable to $U_i$.
	Hence, each $U_i'$ is monotone 0-deformable. Since $W_0$ lives in the complement to the $(N-2)$-skeleton, it is monotone 0-deformable. The set $W_1$ can be chosen monotone deformable to
	the monotone 0-deformable set $W_1\cap X^{(N-2)}\subset X^{(N-2)}\setminus X^{(N-3)}$. Thus, $W_1$ is monotone 0-deformable and so on. As the result we obtain that the set $U'_N=W_0\cup\dots\cup W_{N-1}$ is monotone 0-deformable. In the proof of inductive step the same argument shows that the set
	$U_m=W_0\cup\dots\cup W_{N-1}$ is monotone 0-deformable.
\end{proof}

\subsection{Borel  construction}
Let a group $\pi$ act on spaces $X$ and $E$ with the projections
onto the orbit spaces $q_X:X\to X/\pi$ and $q_E:E\to E/\pi=B$. Let
$q_{X\times E}:X\times E\to X\times_\pi E=(X\times E)/\pi$ denote the
projection onto the orbit space of the diagonal action of $\pi$ on
$X\times E$. Then there is a commutative diagram called the {\em
Borel construction}~\cite{Bo}:
\[
\begin{CD}
X @<pr_X<< X\times E @>pr_2>> E\\
@Vq_XVV @VqVV @Vq_EVV\\
X/\pi @<p_E<< X\times_{\pi} E @>p_X>> B.\\
\end{CD}
\]

If $\pi$ is discrete and the actions are free and proper, then all projections
in the diagram are locally trivial bundles with the structure group $\pi$.  Then the fiber of $p_X$ is homeomorphic to $X$ and the fiber of $p_E$ is homeomorphic to $E$.
For any invariant subset $Q\subset X$ the map $p_X$ defines the pair of bundles
$p_X:(X\times_{\pi}E,Q\times_{\pi}E)\to B$ with the stratified fiber $(X,Q)$
and the structure group $\pi$. 

If $X/\pi$ and $B$ are CW complexes for proper free actions of discrete group $\pi$, their CW structures  define a natural
CW structure on $X\times_{\pi}E$ as follows: First, $X$ and $E$ being covering spaces inherit CW structures from $X/\pi$ and $B$ respectively. Since the diagonal action of $\pi$ on $X\times E$
preserves the product CW complex structure on $X\times E$ and takes cells to cells homeomorphically, the orbit space $X\times_{\pi}E$ receives the induced CW complex structure.

\begin{lemma}\label{cat}
Let $\Wi X$ be the universal covering of an $n$-dimensional simplicial complex $X$ with the fundamental group $\pi=\pi_1(X)$. Suppose that the universal covering admits a classifying map $u:X\to B$ to a $d$-dimensional simplicial complex, $\pi_1(B)=\pi$. Let $E$ be the universal covering of $B$.
Then for the $n$-skeleton $$\cat(\Wi X\times_{\pi}E)^{(n)}\le\frac{d+n}{2}$$
where the CW complex structure on $\Wi X\times_{\pi}E$ is defined by the simplicial complex structures on $X$ and $B$.
\end{lemma}
\begin{proof}
	Denote by $K=\Wi X\times_{\pi}E$. Since
	$(\Wi X\times E)^{(n)}=\bigcup_j\Wi X^{(n-j)}\times E^{(j)}$, we have
	$$K^{(n)}=\bigcup_{j=0}^d\Wi X^{(n-j)}\times_{\pi}E^{(j)}.$$
	We show that $\cat K^{(n)}\le d+\lfloor\frac{n-d}{2}\rfloor=\lfloor\frac{d+n}{2}\rfloor$.

Let  
$m=\lfloor(d+n)/2\rfloor+1$. We apply Proposition~\ref{cover} to $B$ with $r=0$ to obtain an open
cover $\scr U=\{U_1,\dots,U_m\}$ by monotone 0-deformable in $B$ sets with
$\Ord_x\scr U\ge m-j$ for $x\in B^{(j)}$. We note that we apply Proposition~\ref{cover} here
with $r=0$ and $N=d+1$. Thus, we need to be sure that  $m\ge d+1$ which is satisfied since $d\le n$. The substitution
$i=k-1$ helps to see the inequality $\Ord_x\scr U\ge m-i$ for $x\in B^{(j)}$ for $x\in B^{(j)}$.

Since $m>\frac{d+n}{2}$, we have $2m-1>d+n-1$ and hence, $2m-1\ge n=\dim X$. Hence we can apply Proposition~\ref{cover} with $N=m$ and $r=1$ to get an open cover $\scr V=\{V_1,\dots V_m\}$ of $X$ by  1-deformable in $X$ sets such that the restriction of $\scr V$ to $X^{(2j-1)}$ is a $j$-cover,
$j=1,\dots,k$, where $k$ be the smallest integer satisfying the inequality $n\le 2k-1$.

For every $i\le m$ we define  $$W_i=p_E^{-1}(V_i)\cap p_{\tilde X}^{-1}(U_i).$$
We claim that the collection of sets $\{W_1,\dots, W_m\}$ covers $K^{(n)}$.  Let $x\in\Wi X^{(n-j)}\times_{\pi}E^{(j)}$. Then the point $p_{\tilde X}(x)\in B^{(j)}$ is covered
by at least $m-j$ sets $U_{k_1},\dots,U_{k_{m-j}}\in\scr U$. Since $\scr V$ restricted to
$X^{(2(m-j)-1)}$ is a $(m-j)$-cover,
the sets $V_{k_1},\dots,V_{k_{m-j}}$ cover $X^{(2(m-j)-1)}$. 
Note that $2(m-j)-1\ge d+n+2-2j-1\ge n-j$.
Therefore, the point $p_E(x)\in X^{(n-j)}$ is covered by  $V_{k_s}$
for some $s\in\{1,\dots,m-j\}$. Hence, $x\in W_{k_s}$.

We note that $W_i=Q_i\times_{\pi}P_i\subset\Wi X\times_{\pi}E$ where
$P_i=q_B^{-1}(U_i)$ and $Q_i=q_X^{-1}(V_i)$. Thus, its intersection with $K^{(n)}$ can be written as
$$W_i(n)=W_i\cap K^{(n)}=\bigcup_j^dQ_i(n-j)\times_{\pi}P_i(j)$$ where $P_i(k)=P_i\cap E^{(k)}$
and $Q_i(\ell)=Q_i\cap\Wi X^{(\ell)}$.

To complete the proof we show that each set $W_i(n)$ is contractible in $K^{(n)}$.
We consider a monotone deformation $h_t:U_i\to B$ of $U_i$ to $B^{(0)}$. Let $\tilde h_t:P_i\to E$ be the lifting of $h_t$. Thus, $\tilde h_t$ is a $\pi$-equivariant deformation of $P_i$ to $E^{(0)}$.
Then $1_{\Wi X}\times h_t:\Wi X\times P_i\to \Wi X\times E$ is a $\pi$-equivariant deformation and, hence, it defines a deformation of the orbit space
$\bar h_t:\Wi X\times_{\pi}P_i\to K$ which is a lift of $h_t$ with respect to $p_{\Wi X}$. Since each skeleton $\Wi X^{(i)}$ is $\pi$-invariant, the deformation
$\bar h_t$ preserves the filtration of  the fibers $\Wi X$ of the bundle $p_{\Wi X}$
by the skeleta. By the same reason, $\bar h_t$ moves the set $Q_i(n-j)\times_{\pi}P_i$ within $Q_i(n-j)\times_{\pi}B$. Since $h_t$ is monotone,
$\bar h_t$ moves  $Q_i(n-j)\times_{\pi}P^{(j)}$ within $Q_i(n-j)\times_{\pi}B^{(j)}\subset K^{(n)}$
for all $j$. Thus, $\bar h_t$ deformes $W_i(n)$ within $K^{(n)}$ to the set $$Q_i\times_{\pi}E^{(0)}\subset \Wi X\times_{\pi}E^{(0)}=p_{\Wi X}^{-1}(B^{(0)})\cong\coprod_{b\in B^{(0)}}\Wi X.$$

Since $V_i$ is 1-deformable in $X$, so is $Q_i$ in $\Wi X$. Since $\Wi X$ is simply connected,
$Q_i$ is contractible in $\Wi X$. Thus, we obtain that the set
$$Q_i\times_{\pi}E^{(0)}\cong\coprod_{b\in B^{(0)}} Q_i\subset \coprod_{b\in B^{(0)}}\Wi X$$ 
is 0-deformable in $\Wi X\times_{\pi}E^{(0)}\subset K^{(n)}$. Therefore, $W_i(n)$ is 0-deformable in $K^{(n)}$. Since $K$ is connected, $W_i(n)$ is contractible in $K^{(n)}$.

Thus, $\cat K^{(n)}\le m-1=\lfloor\frac{d+n}{2}\rfloor\le\frac{d+n}{2}$.
\end{proof}

\

\section{Main Result}

\begin{thm}\label{main}
	For every simplicial complex $X$ there is the inequality
	$$
	\cat X\le\frac{\cat(u_X)+\dim X}{2}
	$$
	where $u_X:X\to B\pi$ is a classifying map for the universal cover of $X$.
\end{thm}
\begin{proof} Let $\dim X=n$ and $\cat(u_X)=d$.
In the proof we use the notations $B=B\pi$, $B^d=B\pi^{(d)}$ and
$E=E\pi$, $E^d=E\pi^{(d)}$. By Proposition~\ref{classifying} we may assume that the map
$u_X$ lands in $B^d$.
Consider the diagram generated by the Borel construction
$$
\begin{CD}
X @<p_E<< \Wi X\times_{\pi}E @>p_{\Wi X}>> B\\
@A=AA @A\subset AA @A\subset AA\\
X @<p_{E^d}<< \Wi X\times_{\pi}E^d @>p_{\Wi X}|>> B^d.\\
\end{CD}
$$	
Since $E$ is contractible, the map $p_E$ is a homotopy equivalence. Let $g$ be its homotopy inverse. Applying the homotopy lifting property we may assume that $g$ is a section of $p_E$.
Then the map $p_{\Wi X}$ is homotopic to $p_{\Wi X}\circ g\circ p_E$. Note that
the map $p_{\Wi X}\circ g:X\to B$ is a classifying map for $X$. Thus, it is homotopic to 
the map $u_X:X\to  B$ whose  image is in $B^d$. Therefore, $p_{\Wi X}:\Wi X\times_{\pi}E\to B$ is homotopic to a map with image in $B^d$.
Let $p_t:\Wi X\times_{\pi}E\to B$ be such a homotopy. Thus, $p_0=p_{\Wi X}$ and $p_1(\Wi X\times_{\pi}E)\subset B^d$. Let $\bar p_t:\Wi X\times_{\pi}E\to\Wi X\times_{\pi}E$ be the lift
of $p_t$ with $\bar p_0=id$. Then $\bar p_1(\Wi X\times_{\pi}E)\subset \Wi X\times_{\pi}E^d$.

First, we note that $s=\bar p_1\circ g:X\to\Wi X\times_{\pi}E^d$ is a homotopy section of $p_{E^d}$.
Indeed, the homotopy $h_t=p_E\circ\bar p_t\circ g:X\to X$ is joining $h_0=p_E\circ\bar p_0\circ g=p_E\circ g=1_X$
with $h_1=p_E\circ\bar p_1\circ g=p_{E^d}\circ\bar p_1\circ g=p_{E^d}\circ s$.

We may assume that $B$ is a simplicial complex. Denote by $K=\Wi X\times_{\pi}E^d $.
We consider the CW complex structure on $K$ defined by the simplicial complex structures on $X$ and $B$. Next we show that the restriction $(p_{E^d})|_{K^{(n)}}:K^{(n)}\to X$ is a homotopy domination.
Since $\dim X=n$, there is a homotopy $s_t:X\to K$ with $s_0=s$ and $s_1(X)\subset K^{(n)}$.
Then the homotopy $q_t=p_{E^d}\circ s_t:X\to X$ joints $q_0=p_{E^d}\circ s\sim 1_X$ with
$q_1=p_{E^d}\circ s_1=(p_{E^d})|_{K^{(n)}}\circ s_1$. 

Therefore, by Proposition~\ref{domination}, $\cat X\le\cat K^{(n)}$. Lemma~\ref{cat}
implies $$\cat X\le \frac{d+n}{2}.$$
\end{proof}
\begin{cor}\label{cor}
	For any CW complex $X$,
	$$
	\cat X\le\frac{\cd(\pi_1(X))+\dim X}{2}.$$
\end{cor}
\begin{proof}
We note that every CW complex is homotopy equivalent to a simplicial complex of the same dimension. By the Eilenberg-Ganea theorem $\pi=\pi_1(X)$ has a classifying complex $B\pi$
of dimension equal $\cd(\pi)$ whenever $\cd(\pi)\ne 2$ (see~\cite{Br}). 
Thus, If $\cd(\pi)\ne 2$, the result immediately follows from Theorem~\ref{main}.

In the case when $\cd(\pi)=2$
one can find a classifying complex $B\pi$ of dimension three~\cite{Br}. Then Obstruction Theory implies that there
is a map $r:B\pi\to B\pi^{(2)}$ which is the identity on the 1-skeleton. It is easy to check that $r$ induces an isomorphism of the fundamental groups:
Obviously it is surjective and the kernel of $r_*:\pi_1(B)\to\pi_1(B\pi^{(2)})$ is trivial. In particular, its composition with a classifying map
$r\circ u_X: X\to B\pi^{(2)}$ is a classifying map and we can apply Theorem~\ref{main} to it.
\end{proof}

\begin{thm} For any locally trivial bundle $p:E\to B$ with a simply connected fiber $F$ and an aspherical base $B$,
$$
\cat E\le \dim B+\frac{\dim F}{2}.$$
\end{thm}
\begin{proof} By Corollary~\ref{cor}
$$
\cat E\le\frac{cd(\pi_1(E))+\dim E}{2}=\frac{cd(\pi_1(B))+\dim B+\dim F}{2}\le$$
$$\le\frac{2\dim B+\dim F}{2}=\dim B+\frac{\dim F}{2}.$$
\end{proof}
When $B$ is an aspherical manifold we obtain an upper bound $$\cat E\le\cat B+\frac{\dim F}{2}.$$ Therefore for every aspherical $n$-manifold $M$
the LS-category of the total manifold of an $S^3$-fibration $f:N\to M$ is at most $n+1$. All other known result produces  the bound $n+2$ just
in view of the fact that $N$ is inessential. A concrete example would be the total space $N$ of the pull-back of the Hopf bundle $h:S^7\to S^4$ via
an essential map of a 4-torus $g:T^4\to S^4$. I don't see how to get our estimate $\cat N\le 5$ by any other means.

In the case when additionally $\cat F=\frac{\dim F}{2}$, like for $F=\C P^n$,  we have a Hurewicz type formula for $\cat$:
$$
\cat E\le\cat B+\cat F.
$$
We recall that for general fibrations the Hurewicz type formula does not hold. The best known estimate for general locally trivial bundles is
$\cat E\le(\cat B+1)(\cat F+1)-1$ \cite{CLOT}. Note that fibrations with the fiber $\C P^n$ can be produced by projectivization of the spherical bundles of complex vector bundles.

\subsection{r-connected universal cover} We recall a classical result that for $r$-connected space $n$-dimensional complex $X$,
$$
\cat X\le\frac{n}{r+1}.
$$
If $X=B\times Y$ with $r$-connected $Y$, we have $$\cat X\le\cat B+\frac{\dim Y}{r+1}=
\cat B+\frac{n-\dim B}{r+1}$$
$$\le\cat B+\frac{n-\cat B}{r+1}=\frac{r\cat B+n}{r+1}.$$
Below we obtain a similar estimate for general $X$.

In the proof of the main result we applied our technical proposition (Proposition~\ref{cover})
with $r=0$ and $r=1$. Using Proposition~\ref{cover} with $r=0$ and arbitrary $r>0$ brings the
following

\begin{lemma}\label{cat2}
	Suppose that $\Wi X$  the universal covering of an $n$-dimensional simplicial complex $X$ with the fundamental group $\pi=\pi_1(X)$ is $r$-connected. Assume that $\Wi X$ admits a classifying map to $d$-dimensional complex $B$, $\pi_1(B)=\pi$. Let $E$ be the universal covering of $B$.
	Then $$\cat(\Wi X\times_{\pi}E)^{(n)}\le\frac{rd+n}{r+1}.$$
\end{lemma}
This  Lemma brings the following generalization of Theorem~\ref{main}.
\begin{thm}\label{main2}
	For every simplicial complex $X$  with $r$-connected universal cover $\Wi X$ there is the inequality
	$$
	\cat X\le\frac{r\cat(u_X)+\dim X}{r+1}
	$$
	where $u_X:X\to B\pi$ is a classifying map for the universal cover of $X$.
\end{thm}

\begin{cor}
	For any CW complex $X$ with $r$-connected universal covering $\Wi X$,
	$$
	\cat X\le\frac{r\cd(\pi_1(X))+\dim X}{r+1}.$$
\end{cor}



\end{document}